\documentclass[12pt]{amsart}

\usepackage{latexsym, amssymb, amsmath}

\usepackage{amsfonts, graphicx}
\usepackage{graphicx,color}
\newcommand{\be}{\begin{equation}}
\newcommand{\ee}{\end{equation}}
\newcommand{\beq}{\begin{eqnarray}}
\newcommand{\eeq}{\end{eqnarray}}

\newtheorem{thm}{Theorem}[section]

\newtheorem{lma}{Lemma}[section]

\newtheorem{cor}{Corollary}[section]

\theoremstyle{remark}
\newtheorem{rem}{Remark}[section]
\numberwithin{equation}{section}

\def\p{\partial}

\def\C{\mathcal{C}}
\def\R{\mathbb{R}}

\def\p{\partial}
\def\lf{\left}
\def\ri{\right}
\def\e{\epsilon}
\def\ol{\overline}
\def\R{\Bbb R}
\def\wt{\widetilde}
\def\la{\langle}
\def\ra{\rangle}

\def\Ric{\text{\rm Ric}}

\def\Pi{\overline{\displaystyle{\mathbb{II}}}}

\def\a{\alpha}
\def\b{\beta}

\def\heat{\lf(\Delta -\frac{\p}{\p t}\ri)}
\def\K{K\"ahler }

\def\heat{\lf(\frac{\p}{\p t}-\Delta\ri)}

\def\Ric{\text{Ric}}
\def\lf{\left}
\def\ri{\right}

\def\a{\alpha}
\def\ol{\overline}

\def\e{\epsilon}
\def\p{\partial}

\def\C{\Bbb C}
\def\R{\Bbb R}

\def\bb{{\bar{\beta}}}

\def\Ric{\operatorname{Ric}}
\def\Rm{\operatorname{Rm}}

\def\K{K\"ahler\ }
\def\be{\begin{equation}}
\def\ee{\end{equation}}

\def\b{\bar}

\def\ppt{\frac{\partial}{\partial t}}

\def\bee
{\begin{equation*}}

\def\eee{\end{equation*}}

\def\la{\langle}
\def\ra{\rangle}
\def\bee{\begin{equation*}}
\def\eee{\end{equation*}}

\def\ol{\overline}
\def\e{\epsilon}
\def\lf{\left}
\def\heat{\lf(\frac{\p}{\p t}-\Delta\ri)}

\def\ri{\right}

\def\a{{\alpha}}
\def\b{{\beta}}

\def\bb{{\bar\b}}

\def\wt{\widetilde}

\def\p{\partial}

\def\K{K\"ahler }
\def\KR{K\"ahler-Ricci }

\def\be{\begin{equation}}
\def\ee{\end{equation}}
\def\ol{\overline}
\def\lf{\left}
\def\ri{\right}
\def\a{{\alpha}}
\def\b{{\beta}}

\def\e{\epsilon}

\def\bb{{\bar\b}}

\def\Ric{\text{\rm Ric}}
\def\Rm{\text{\rm Rm}}

\def\wt{\widetilde}

\def\p{\partial}

\def\C{\Bbb C}

\def\wt{\widetilde}

\def\p{\partial}

\def\p{\partial}

\def\C{\Bbb C}

\def\ppt{\frac{\partial}{\partial t}}

\def\la{\langle}
\def\ra{\rangle}
\def\bee{\begin{equation*}}
\def\eee{\end{equation*}}

\def\ol{\overline}
\def\e{\epsilon}
\def\lf{\left}
\def\heat{\lf(\frac{\p}{\p t}-\Delta\ri)}

\def\ri{\right}

\def\a{{\alpha}}
\def\b{{\beta}}

\def\bb{{\bar\b}}

\def\wt{\widetilde}

\def\p{\partial}

\def\K{K\"ahler }
\def\KR{K\"ahler-Ricci }

\def\be{\begin{equation}}
\def\ee{\end{equation}}
\def\ol{\overline}
\def\lf{\left}
\def\ri{\right}
\def\a{{\alpha}}
\def\b{{\beta}}

\def\e{\epsilon}

\def\bb{{\bar\b}}

\def\Ric{\text{\rm Ric}}
\def\Rm{\text{\rm Rm}}

\def\wt{\widetilde}

\def\p{\partial}

\def\C{\Bbb C}

\def\wt{\widetilde}

\def\p{\partial}

\def\p{\partial}

\def\C{\Bbb C}

\def\ppt{\frac{\partial}{\partial t}}

\def\la{\langle}
\def\ra{\rangle}
\def\bee{\begin{equation*}}
\def\eee{\end{equation*}}

\def\ol{\overline}
\def\e{\epsilon}
\def\lf{\left}
\def\heat{\lf(\frac{\p}{\p t}-\Delta\ri)}

\def\ri{\right}

\def\a{{\alpha}}
\def\b{{\beta}}

\def\bb{{\bar\b}}

\def\wt{\widetilde}

\def\p{\partial}

\def\K{K\"ahler }
\def\KR{K\"ahler-Ricci }

\def\be{\begin{equation}}
\def\ee{\end{equation}}
\def\ol{\overline}
\def\lf{\left}
\def\ri{\right}
\def\a{{\alpha}}
\def\b{{\beta}}

\def\e{\epsilon}

\def\bb{{\bar\b}}

\def\Ric{\text{\rm Ric}}
\def\Rm{\text{\rm Rm}}

\def\wt{\widetilde}

\def\p{\partial}

\def\C{\Bbb C}

\def\wt{\widetilde}

\def\p{\partial}

\def\p{\partial}

\def\C{\Bbb C}

\def\ppt{\frac{\partial}{\partial t}}

\def\la{\langle}
\def\ra{\rangle}
\def\bee{\begin{equation*}}
\def\eee{\end{equation*}}

\def\ol{\overline}
\def\e{\epsilon}
\def\lf{\left}
\def\heat{\lf(\frac{\p}{\p t}-\Delta\ri)}

\def\ri{\right}

\def\a{{\alpha}}
\def\b{{\beta}}

\def\bb{{\bar\b}}

\def\wt{\widetilde}

\def\p{\partial}

\def\K{K\"ahler }
\def\KR{K\"ahler-Ricci }

\def\be{\begin{equation}}
\def\ee{\end{equation}}
\def\ol{\overline}
\def\lf{\left}
\def\ri{\right}
\def\a{{\alpha}}
\def\b{{\beta}}

\def\e{\epsilon}

\def\bb{{\bar\b}}

\def\Ric{\text{\rm Ric}}
\def\Rm{\text{\rm Rm}}

\def\wt{\widetilde}

\def\p{\partial}

\def\C{\Bbb C}

\def\wt{\widetilde}

\def\p{\partial}

\def\p{\partial}

\def\C{\Bbb C}

\def\ppt{\frac{\partial}{\partial t}}

\begin{document}

\title[]
{On existence and curvature estimates of Ricci flow}

 \author{Man-Chun Lee}
\address[Man-Chun Lee]{Department of
 Mathematics, The Chinese University of Hong Kong, Shatin, Hong Kong, China.}
\email{mclee@math.cuhk.edu.hk}

\author{Luen-Fai Tam$^1$}
\address[Luen-Fai Tam]{The Institute of Mathematical Sciences and Department of
 Mathematics, The Chinese University of Hong Kong, Shatin, Hong Kong, China.}
 \email{lftam@math.cuhk.edu.hk}
\thanks{$^1$Research partially supported by Hong Kong RGC General Research Fund \#CUHK 14305114}

\renewcommand{\subjclassname}{
  \textup{2010} Mathematics Subject Classification}

\date{April, 2017}

\begin{abstract} In this work, using the method by He \cite{He2016}, we prove a short time existence for Ricci flow on a complete noncompact Riemannian manifold with the following properties: (i) there is $r_0>0$ such that the volume of any geodesic balls of radius $r\le r_0$ is close to the volume of the geodesic ball in the Euclidean space with the same radius; (ii) the Ricci curvature is bounded below; (iii) the sectional curvature is bounded below by $-Cd^2(x,x_0)$ where $d(x,x_0)$ is the distance from a fixed point. An application to the uniformization for complete noncompact \K manifolds with nonnegative bisectional curvature and maximal volume growth is given. The result is related to the results of He \cite{He2016} and Liu \cite{Liu2015,Liu2016}. We   prove if  $g(t)$ is a complete solution  to the \KR flow on a complete noncompact \K manifold so that $g(0)$ has nonnegative bisectional curvature and the curvature of $g(t)$ is bounded by $at^{-\theta}$ for some $a>0, 2>\theta>0$, then $g(t)$ also has nonnegative bisectional curvature. This generalizes a previous result in \cite{HuangTam2015}. Under the assumption that $g(0)$ is noncollapsing, with additional assumption that the curvature of $g(t)$ is uniformly bounded away from $t>0$, we prove that the curvature of $g(t)$ is in fact bounded by $at^{-1}$ for some $a>0$. This result is proved using  the method by Simon and Topping \cite{SimonTopping2016}.

\end{abstract}


\maketitle

\markboth{Man-Chun Lee and Luen-Fai Tam}{Existence and curvature estimates of Ricci flow}
\section{introduction}

In \cite{Y}, Yau made the following conjecture:
 {\it A complete noncompact
K\"ahler manifold of complex dimension $m$ with positive holomorphic bisectional curvature is
biholomorphic to $\C^{m}$.} There are many results related to this conjecture. For a survey, see \cite{ChauTam2008} for example. For  the case that the manifold has maximal volume growth, there are some partial  results.
Let $(M^m,g)$ be a complete noncompact \K manifold of complex dimension $m$ with nonnegative holomorphic bisectional curvature and with maximal volume growth.   It was proved by Chau and the second author \cite{ChauTam2008}  that Yau's conjecture is true if   the curvature is bounded.  Without assuming the curvature is bounded, recently it is proved that $M$ is also biholomorphic to $\C^m$ in the following cases:

\begin{enumerate}
  \item [(i)] $M$ has nonnegative complex sectional curvature. Proved by Huang and the second author \cite{HuangTam2015}  using the result of Cabezas-Rivas and Wilking \cite{Cabezas-RivasWilking2011}.
  \item [(ii)] Either $m\le3$ or the asymptotical volume ratio satisfies $$\lim_{r\to\infty}\frac{V(r)}{r^{n}}\ge (1-\e)\omega_{n}$$
      provided $\e>0$ is small enough depending only on $n=2m$, where $\omega_{n}$ is the volume of the unit ball in $\R^{n}$. Proved by Liu \cite{Liu2015, Liu2016}.
  \item [(iii)] There exists $r_0>0$ such that $|\p\Omega|^n\ge (1-\e)I_n |\Omega|^{n-1}$ for some $\e>0$ depending only on $n=2m$ for all $\Omega \subset B(x,r_0)$ for all $x$, where $I_n$ is the optimal isoperimetric constant of $\R^n$. Proved by He \cite{He2016}.
\end{enumerate}

The condition on asymptotical volume ratio in (ii) or the condition on the isoperimetric constant in (iii) implies that there is $r_0>0$ such that $V(B(x,r_0))\ge (1-\e)r_0^{n}\omega_n$. In this note we want to study Yau's conjecture under this condition. We will prove the following:
\begin{thm}\label{t-intro-1} There is $\delta_0>0$ depending only on $m$ such that the following is true: Suppose $(M^m,g)$ is a complete noncompact \K manifold with complex dimension $m$, nonnegative bisectional curvature and with maximum volume growth such that $V_g(x,r_0)\ge (1-\delta_0)r_0^{2m}\omega_{2m}$ for some $r_0>0$  for all $x$ and the sectional curvature of $M$ satisfies $K_g(x)\ge -r^2(x) $ where $r(x)$ is the distance to a fixed point. Then $M$ is biholomorphic to $\C^m$.
\end{thm}
In particular, the theorem can be applied to \K manifolds with nonnegative sectional curvature. The lower bound of the sectional curvature can be further relaxed, see Corollary \ref{c-uniformization} for more details.
The idea of proof is to use Ricci flow which was originated by W. S. Shi \cite{Shi1997}. The main point is to deform the  metric to a metric with bounded curvature and with nonnegative bisectional curvature so that one can use the result in \cite{ChauTam2008}. Theorem \ref{t-intro-1} is then a consequence of Theorem \ref{t-shortime} and the result in \cite{HuangTam2015} which says that on a complete solution $g(t)$ to the Ricci flow with curvature bound $|\Rm(g(t))|\le a/t$, then $g(t)$ is \K if $g(0)$ is K\"ahler, and has nonnegative bisectional curvature provided that $a$ is small depending only on the dimension. We  use  the idea of He \cite{He2016} to prove Theorem \ref{t-shortime} which is a short time existence result for Ricci flow. Instead of using the pseudo-locality theorem of Perelman \cite{Perelman-1}(see also \cite{cty}) as by He \cite{He2016},
we will use the following pseudo-locality theorem by Tian-Wang \cite{TW}, see Theorem \ref{t-TW}.

\begin{rem} In a most updated version of \cite{Liu2016}, Liu actually proves Yau's conjecture under the assumption of maximal volume growth. His method is very different.

\end{rem}

 There are solutions $g(t)$ to the \KR flow with the property that $|\Rm(g(t))|\le a/t$ where $a>0$ may not be small, see \cite{ChauLiTam1} for example. So it is desirable to remove the restriction that $a$ is small in a result in \cite{HuangTam2015} on the preservation of nonnegativity of bisectional curvature. In fact, we will prove a stronger result:
\begin{thm}\label{t-intro-2} Let $(M^n,g(t))$  be a smooth complete noncompact solution to the \KR flow on $M\times[0,T]$, $T>0$. Suppose $g(0)$ has nonnegative bisectional curvature and suppose
$|\Rm(g(t))|\le a/t^\theta$ on $M$ for some $a>0$ and for some $0<\theta<2$ for all $t\in (0,T]$. Then $g(t)$ also has nonnegative bisectional curvature for $t\in[0,T]$. \end{thm}

For   three dimensional Riemannian manifolds, it is proved by Chen-Xu-Zhang \cite{ChenXuZhang2013} that nonnegativity of Ricci curvature is preserved   for complete solutions to the Ricci flow, without any assumption on the bound of the curvature. It was proved by Yang-Zheng \cite{YangZheng2013} that any complete $U(n)$ invariant solutions $g(t)$ to the \KR flow on $\C^n$ with $\mathrm{BK}(g(0))\ge0$   then so is $g(t)$ for all $t>0$, without any condition on the bound of curvature.  One may ask to what extent that Yang-Zheng's result is also true for general \KR flow. So we would like to discuss   conditions so that the curvature bound in Theorem \ref{t-intro-2} is true. With the curvature estimates that imply Theorem \ref{t-intro-2} (see Theorem \ref{t-lowerbound-BK-1}),  the method of Simon-Topping \cite{SimonTopping2016} can be carried over  to prove the following:

\begin{thm}\label{t-intro-3} Let $(M^n,g(t)$ be a smooth complete solution to the \K Ricci for for $t\in [0,T]$, $T>0$.
  Suppose further

  \begin{enumerate}
    \item [(i)] $\mathrm{BK}(g(0))\ge0$;
    \item [(ii)] $V_0(x,1)\ge v_0>0$ for all $x\in M$ for some $v_0>0$; and
    \item [(iii)] the curvature of $g(t)$ is uniformly bounded on $M\times(\tau, T]$ for all $\tau>0$.
  \end{enumerate}
     Then $|\Rm(g(t))|\le \displaystyle{\frac at}$  on $M$ for some $a>0$ for all $t\in (0,T]$. In particular, $\mathrm{BK}(g(t))\ge0$ for all $t\in (0,T]$.
\end{thm}
The condition (iii) is needed because we have to use the pseudo-local theorem of Perelman \cite{Perelman-1}.

The paper is organized as follows: Theorem \ref{t-intro-1} will be proved in section 2, Theorem \ref{t-intro-2} will be proved in section 3, and Theorem \ref{t-intro-3} will be proved in section 3.

{\it Acknowledgement}: The authors would like to thank Hung-Hsi Wu for some useful discussions. They would also like to thank Albert Chau and Fei He for their interest in the work.

\section{short time existence of Ricci flow}
Recall the Ricci flow equation on  $(M^n,g_0)$:

\be\label{e-Ricciflow}
\left\{
  \begin{array}{ll}
   \displaystyle{ \frac{\p  g}{\p t}} = &-2\Ric(g) \hbox{\ \ in $M\times[0,T)$;} \\
    g(0)=  &g_0.
  \end{array}
\right.
\ee

In this section, we want to prove the following:

\begin{thm}\label{t-shortime}
For any $\a>0$, there is $1>\delta_0>0$ depending only on $\a, n$ such that if $(M^n,g_0)$ is a complete noncompact Riemannian manifold satisfying the following:
\begin{enumerate}
  \item [(a)] $\Ric(g_0)\ge -L^2$ for some constant $L\ge0$;
  \item [(b)] there is $1>r_0>0$ such that $V_{g_0}(x,r)\ge (1-\delta_0)\omega_n r^n$ for all $x\in M$ for all $r\le r_0$, where $\omega_n$ is the volume of the unit ball in $\R^n$ and $V_{g_0}(x,r)$ is the volume of geodesic ball of radius $r$ with center at $x$ with respect to $g_0$; and
  \item [(c)] the sectional curvature satisfies $K_{g_0}(x)\ge -\lf(\phi(r(x))\ri)^2$   where $r(x)$ is the distance function from a fixed point and $\phi:[0,\infty)\to \R$ is a smooth function so that $\phi(s)\ge C>0$ for some constant $C$, $\phi'\ge0$, $\phi(s+1)\le C\phi(s)$ for some $C>0$ for all $s$ large enough, and
      $$
      \int_0^\infty \frac{ds}{\phi(s)}=\infty.
      $$
\end{enumerate}
Then the Ricci flow \eqref{e-Ricciflow} has a complete solution on $M\times[0,\e_0^2)$ for some $\e_0$ depending only on $\a, n, L, \phi, r_0$ such that
\begin{enumerate}
  \item [(i)] the curvature of $g(t)$ satisfies:
  $$|\Rm(g(t))|\le \frac\a t;$$
  and
  \item [(ii)] $V_{g(t)}(x,\sqrt t)\ge Ct^\frac n2$ for some positive constant $C$ depending only on $\a, n, l, \phi, r_0$ for all $x\in M$ and $t\in [0,\e_0^2)$.
\end{enumerate}

\end{thm}
Here are some examples of $\phi$ satisfying the condition in the theorem .   $$\phi(s)=s(\log s)(\log^{(2)}s)\cdots(\log ^{(k)}s)$$ for $s$ large enough. Here $\log^{(k)}s=\log(\log^{(k-1)}s)$.

\begin{cor}\label{c-uniformization} There is $\delta_0>0$ depending only on $m$ such that the following is true: Suppose $(M^m,g)$ is a complete noncompact \K manifold with complex dimension $m$, nonnegative bisectional curvature and with maximum volume growth such that $V_g(x,r_0)\ge (1-\delta_0)r_0^{2m}\omega_{2m}$ for some $r_0>0$  for all $x$ and the sectional curvature of $M$ satisfies $K_g(x)\ge - \lf(\phi(r(x))\ri)^2$ where $\phi(s)$ is as in Theorem \ref{t-shortime} and $r(x)$ is the distance to a fixed point. Then $M$ is biholomorphic to $\C^m$.
\end{cor}
\begin{proof}  Let $\a(m)>0$ to be chosen. Let $\delta_0>0$ be the constant in Theorem \ref{t-shortime} for such  $\a>0$ and $m$. By the theorem, if $g_0$ satisfies the conditions in the corollary, then the Ricci flow has a complete short time solution $g(t)$ with curvature $|\Rm(g(t))|\le \a t^{-1}$. By \cite{HuangTam2015}, $g(t)$ is K\"ahler.  We can choose $\a>0$ small enough depending only on $n$ so that $g(t)$ also has nonnegative bisectional curvature by \cite{HuangTam2015}. Moreover, one can prove that $g(t)$ has maximum volume growth, see \cite{He2016}. Hence $M$ is biholomorphic to $\C^m$ by \cite{ChauTam2008}.

\end{proof}

%\section{some basic facts}

Before we prove Theorem \ref{t-shortime}, we need some basic facts.
As in \cite{He2016,Hochard2016,Topping2010},
for $0<\kappa<1$, let $f:[0,1)\to[0,\infty)$ be the function:
\be\label{e-exh-1}
 f(s)=\left\{
  \begin{array}{ll}
    0, & \hbox{$s\in[0,1-\kappa]$;} \\
    -\displaystyle{\log \lf[1-\lf(\frac{ s-1+\kappa}{\kappa}\ri)^2\ri]}, & \hbox{$s\in (1-\kappa,1)$.}
  \end{array}
\right.
\ee
Let   $\varphi\ge0$ be a smooth function on $\R$ such that $\varphi(s)=0$ if $s\le 1-\kappa+\kappa^2 $, $\varphi(s)=1$ for $s\ge 1-\kappa+2 \kappa^2 $
\be\label{e-exh-2}
 \varphi(s)=\left\{
  \begin{array}{ll}
    0, & \hbox{$s\in[0,1-\kappa+\kappa^2]$;} \\
    1, & \hbox{$s\in (1-\kappa+2\kappa^2,1)$.}
  \end{array}
\right.
\ee
such that $\displaystyle{\frac2{ \kappa^2}}\ge\varphi'\ge0$. Define
 $$F(s):=\int_0^s\varphi(\tau)f'(\tau)d\tau.$$
 The following has been proved by He \cite{He2016}, see also \cite{Hochard2016,Topping2010}.

\begin{lma}\label{l-exhaustion-1} Suppose   $0<\kappa<\frac18$. Then the function $F\ge0$ defined above is smooth and satisfies the following:
\begin{enumerate}
  \item [(i)] $F(s)=0$ for $0\le s\le 1-\kappa+\kappa^2$.
  \item [(ii)] There exists an absolute constant $c_1>0$ such that $$0\le e^{-F(s)}F'(s)\le\frac {c_1}\kappa;\ \   0\le e^{-2F}F''(s)\le \frac{c_1}{\kappa^2}.$$
  \item [(iii)]  For any $ 1-2\kappa <s<1$, there is $\tau>0$ with $0<s -\tau<s +\tau<1$ such that
 \bee
 1\le \exp(F(s+\tau)-F(s-\tau))\le (1+c_2\kappa);\ \ \tau\exp(F(s_0-\tau))\ge c_3\kappa^2
 \eee
  for some absolute constants  $c_2>0, c_3>0$.
\end{enumerate}

\end{lma}

\begin{proof} We sketch the proof here. The fact that $F$ is smooth, $F\ge0$ and satisfies (i), (ii) follows by   direct computations and the fact that
\be\label{e-phi}
\begin{split}
1\le  \exp\lf(\int_0^s(1-\varphi(\tau))f'(\tau)d\tau\ri) \le \frac{1-\kappa^2}{1-4\kappa^2}.
\end{split}
\ee

To prove (iii), let
 $$
 B:=\exp(F(s+\tau)-F(s-\tau))
 $$
 and
 $$
 A:=\tau \exp(F(s-\tau))=\tau B^{-1}\exp(F(s+\tau))\ge \tau B^{-1}.
 $$

 Suppose $1-2\kappa\le s\le 1-\frac23\kappa$, let $\tau=\kappa^2$. Then $1>s+\tau>s-\tau>0$ and
\bee
\begin{split}
F(s+\tau)-F(s-\tau)= \int_{s-\tau}^{s+\tau}F'(t)dt
\le  2\tau F'(s+\tau)
\le  \frac{4\kappa(\frac13+\kappa)}{1-(\frac13+\kappa)^2}
\le 4\kappa.
\end{split}
\eee
So
$
B\le \exp(4\kappa)\le 1+8\kappa,
$
and
$
A\ge\kappa^2 B^{-1}\ge \kappa^2(1+8\kappa)^{-1}.
$
Hence in this case, (iii) is true by \eqref{e-phi}.

Suppose $1>s\ge 1-\frac23\kappa$, then for $0\le \tau\le \frac12(1-s)$, we have
 $s +\tau<1$ and
  $s -\tau\ge \frac12 (3s-1) \ge 1-\kappa>0.
  $

\bee
\begin{split}
B=  &1+\displaystyle{\frac{ 4\tau\kappa^{-2}(s-1+\kappa) }
{ 1-\lf(\frac{  s+\tau -1+\kappa}{\kappa}\ri)^2 }}
:= 1+\psi(\tau).\\
 %=&1+\displaystyle{\frac{ 2\tau (\kappa-2\tau) }{ (s_0+\tau-1+2\kappa)(1-s_0-\tau) }}\\
 %=&1+\frac{2(\kappa-2\tau)}{2\kappa-\tau}
\end{split}
\eee
$\psi$ is continuous on $[0,\frac12(1-s)]$,
with
$
\psi(0)=0,
$
and
$$
\psi(\frac12(1-s))=\frac{4(\kappa-1+s)}{2\kappa-\frac12(1-s) }>\frac23\kappa.
$$
Hence there is $0<\tau<\frac12(1-s)$ such that
$
\psi(\tau)= \frac23\kappa,
$
and so
 $B=1+\frac23\kappa.$
On the other hand,
\bee
\begin{split}
A= \frac{\tau B^{-1}}{ 1-\lf(\frac{  s+\tau -1+\kappa}{\kappa}\ri)^2 }
= \displaystyle{\frac{ B^{-1}\psi(\tau) } {4  \kappa^{-2}(s-1+\kappa) }}
\ge  \frac1{6(1+\frac23\kappa)}\kappa^2.
\end{split}
\eee
Hence (iii) is also true.
\end{proof}

Let $(M^n,g)$ be a complete noncompact Riemannian manifold. Let $\eta:M\to \R$ be a continuous function. Following \cite{GreeneWu1979},
a continuous function $f$ on $M$  is said to be {\it $\eta$-convex} if  for any $x\in M$, the function $f$ is $\eta(x)$-convex at $x$ in the sense that there is a neighborhood $U$ of $x$ and $\delta>0$  such  that $F(y)=f(y)-\frac12(\eta(x)+\delta)d^2(x,y)$ is a convex function in $U$ where $r(x,y)$ is the distance between $x, y$. Namely, for any geodesic $\gamma:[-\e,\e]\to U$, $2F(\gamma(0))\le F(\gamma(\e))+F(\gamma(-\e))$. It is easy to see that if $\eta_1\ge \eta_2$, then $f$ is $\eta_2$-convex implies $f$ is $\eta_1$ convex because $d^2(x,y)$ is a convex function near $x$.

\begin{lma}\label{l-distance-1}
Let $(M^n,g)$ be a complete noncompact Riemannian manifold. Let $\phi$ be the function as in Theorem \ref{t-shortime}. Suppose the sectional curvature of $g$ satisfies $K_g(x)\ge -[\phi(r(x))]^2$  where $r(x)=r(x,x_0)$ is the distance function from a fixed point $x_0$. Then $-r(x)$ is $\eta$-convex where $\eta(x)=-C_1 \phi(r(x)) $  outside a compact set for some positive constant $C_1$.
\end{lma}
\begin{proof} This is standard.  We sketch the proof as follows. Suppose $r(x,x_0)\ge 1$. Let $r(x,y)$ be the distance between $x, y$ and  $R=r(x,x_0)$. Then $K_g(y)\ge -2(\phi(R))^2$ for  $y\in B(x_0, R+\delta)$ for some $\delta>0$. There is constant $C_2$ independent of $x$ such that if $z$ is close to $x_0$ then any point $y$ within the cut locus of $z$ satisfies $\text{\rm Hess}(r(z,\cdot))\le C_2 \phi(R)g$ at $y$.

 Suppose $x$ is within the cut locus of $x_0$, then $\text{\rm Hess}(r (y))\le C_2\phi(R)g(y)$ for some constant $C_2$ for all $y$ in a neighborhood of $x$.
Hence $-r(y)$ is $-C_2 \phi(R)$-convex at $x$.

If $x$ is a cut point of $x_0$, and    Suppose $-r(y)$ is not $-C_3\phi(R)$-convex at $x$ with $C_3=C_2+1$. Let $$F(y)=-r(y)-\frac12(-(C_2+\frac12)\phi(R)) r^2(x,y) =-r(y)+\frac12(C_2+\frac12)\phi(R)r^2(x,y).$$
 Then there exist geodesics $\gamma_i:[-\rho_i,\rho_i]$  parametrized by arc length with  $\gamma_i(0)=x_i\to x$, $\rho_i\to0$ but
\be\label{e-distance-1}
2F(\gamma_i(0))=2F(x_i)> F(\gamma_i(-\rho_i))+F(\gamma(\rho_i)).
\ee
Let $0<\e<<1$, and let $\a_i$ be minimal geodesics from $x_0$ to $x_i$. We may assume that $\a_i(\e)\to \a(\e)$ where $\a$ is a minimal geodesic from $x_0$ to $x$. Then $x_0$ is within the cut locus of  $\a(\e)$. Hence  there are open sets $U$ of $x_0$ and $V$ of $\a(\e)$ such that every point of $U$ is within the cut locus of every point of $V$.  In particular every point of $U$ is within the cut locus of $\a_i(\e)$ for $i$ large enough. Consider the function $F_i(y)=-r(\a_i(\e),y)-\e+ \frac12(C_2+\frac12)\phi(R) r^2(x, y)$. This is convex in $U$ provided $U$ is small enough. Then for $i$ large enough.
$$
2F(x_i)=2F_i(x_i)\le F_i(\gamma_i(-\rho_i))+F_i(\gamma(\rho_i))\le F(\gamma_i(-\rho_i))+F(\gamma_i(\rho_i)).
$$
This contradicts \eqref{e-distance-1}.

\end{proof}

By the result of Greene-Wu \cite[Propositions 2.2, 2.3]{GreeneWu1979}, we have:
\begin{cor}\label{c-distance-1}
Let $(M^n,g)$ be a complete noncompact Riemannian manifold. Suppose the sectional curvature of $g$ satisfies $K_g(x)\ge -(\phi(r(x))^2 $ where $\phi(s)$ is a function as in Theorem \ref{t-shortime}, where $r(x)$ is the distance function from a fixed point $x_0$. For any $\e>0$ there is a smooth function $\rho(x)$ such that (i) $|\rho(x)-d(x)|\le \e$ outside a compact set; (ii) $|\nabla\rho|\le 2$; (iii) ${\rm Hess}(\rho)(x)\le C' \phi(\rho(x)+\e)g$ for some positive constant $C'>0$ outside a compact set.
\end{cor}

%\section{Proof of Theorem \ref{t-shortime}}

To prove Theorem \ref{t-shortime} we use the idea of He \cite{He2016}. In our case, instead of using the pseudo-locality theorem of Perelman \cite{Perelman-1} (see also \cite{cty}) as by He \cite{He2016},
we will use the following pseudo-locality theorem by Tian-Wang \cite{TW}:
\begin{thm}\label{t-TW}
For $\alpha>0$, there exist  $\delta >0$,  $\e_0>0$ and $C>0$ depending only on $n, \a$ such that the following holds:

Suppose $(M^n,g(t)),\,t\in [0,T)$ is a complete Ricci flow with bounded curvature, $p\in M$. Suppose
$ V_{g(0)}(B_{g(0)}(p,1)
)\geq (1-\delta) \omega_n $,
and
$ \Ric(g(0)) \geq -\delta^2 g(0)\quad\text{on}\;\;B_0(p,1)$. Then

$$|\Rm|(p,t)\leq \frac{\alpha}{t};  {\text\  and}\ \
 V_{g(t)}(B_{g(t)}(p,\sqrt{t}))\ge Ct^\frac{n}{2}
$$
for $t\in (0,T)\cap (0,\e_0^2)$.
\end{thm}

\begin{proof}[Proof of Theorem \ref{t-shortime}]
Let $(M^n,g_0)$ be as in Theorem \ref{t-shortime} with $\delta_0$ be   chosen so that $1-\delta_0\ge(1-\delta)^\frac12$ where $\delta$ is the constant in Theorem \ref{t-TW}.  Let $\rho(x)$ be the function in Corollary \ref{c-distance-1} so that $|\rho(x)-r(x)|\le 1$ for $r(x)\ge r_0$, say. Let
\bee
\sigma(x)=\exp\lf(\int_0^{\rho(x)}\frac{ds}{\phi(s)}\ri).
\eee
Then $\sigma(x)$ is an exhaustion function of $M$. Moreover,
\be\label{e-exh-3}
|\nabla \sigma|(x)\le C_1\sigma(x),
\ee
and
\be\label{e-exh-4}
\begin{split}
\sigma_{ij}(x)\le &\sigma(x) \frac1{\phi(\rho(x))}\rho_{ij}(x)\\
\le &C_1\sigma(x)
\end{split}
\ee
for some constant on $C_1$ for all $x\in M$.

Choose $\sigma_k\to \infty$ and let $U_k=\{x|\ \sigma(x)<\sigma_k\}$. Let $0<\kappa<\frac18$ be a fixed number such that
 \be\label{e-kappa}
 (1+c_2\kappa)^{-n}\ge (1-\delta)^\frac12
 \ee
 where $c_2>0$ is the absolute constant in Lemma \ref{l-exhaustion-1}. Hence $\delta_0, \kappa$ depend only on $n, \a$.
 Let $F(s)$ be the function as in Lemma \ref{l-exhaustion-1} corresponding to $\kappa$. Define the metric $h_k(x)=\exp(2F_k(x))g_0(x)$, where $F_k(x)=F(\frac{\sigma(x)}{\sigma_k})$. As in \cite{He2016}, $(U_k, h_k)$ is complete with bounded curvature. The Ricci tensor $\wt R_{ij}$ of $h_k$ and the Ricci tensor $R_{ij}$ of $g_0$ are related as follows:

\bee
\begin{split}
\wt R_{ij}=&R_{ij}-(n-2)(F_k)_{ij}+(n-2)(F_k)_i(F_k)_j-(\Delta  F_k +(n-2)|\nabla  F_k |^2)g_{ij}.
\end{split}
\eee
where the derivatives of $F_k$ are taken with respect to $g$. By Lemmas  \ref{l-exhaustion-1} and \ref{l-distance-1}, \eqref{e-exh-3}, \eqref{e-exh-4} one can conclude that
\be\label{e-metric-1}
\wt R_{ij}\ge -L^2h_k
\ee
for some constant $L>0$ depending only on $n, \kappa, l, \phi$.

On the other hand, for any $x\in U_k$, if $\sigma(x)\le (1-2\kappa)\sigma_k$, then $\sigma(y)\le (1-\kappa+\kappa^2)\sigma_k<\sigma_k$ in $B_{g_0}(x,r_1)$ for some $r_0\ge r_1>0$ independent of $k, x$. So $F_k(y)=0$ in $B_{g_0}(x,r_1)$. Hence $(B_k(x,r_1),h_k)=(B_{g_0}(x,r_1),g_0)$. From this we conclude that $V_k(B_k(x,r))\ge (1-\delta)r^n\omega_n$ for all $r\le r_1$, where $B_k(x,r)$ is the geodesic ball of radius $r$ with respect to $h_k$ and $V_k$ is the volume with respect to $h_k$.

Suppose $\sigma(x)\ge (1-2\kappa)\sigma_k$, let $s=\sigma(x)/\sigma_k<1$ and let $\tau>0$ be as in Lemma \ref{l-exhaustion-1}(iii). Let $r_2$ be the largest $r$ so that $B_k(x,r)\subset \{y|\ s-\tau<\sigma(y)/\sigma_k<s+\tau\}$. Then there is $y\in \p B_k(x,r_2)$ with $\sigma(y)=(s+\tau)\sigma_k$ or $(s-\tau)\sigma_k$. Now
$$
r_2\ge  \exp(F(s-\tau))\ell
$$
where $\ell$ is the length of the minimal geodesic $\gamma(t)$ (with $0\le t\le 1$) from $x$ to $y$ with respect to $h_k$, while the length is measured with respect to $g_0$. Now
\bee
\begin{split}
\tau\sigma_k=&|\sigma(x)-\sigma(y)|\\
=&\int_0^1\la\nabla^{g_0}\sigma, \gamma'\ra_{g_0}dt\\
\le &C_1\sigma_k \ell
\end{split}
\eee
by \eqref{e-exh-3}. Hence we have
$$
r_2\ge C_2\kappa^2
$$
for some constants $C_2>0$ independent of $x, k$ by Lemma \ref{l-exhaustion-1}. Hence  for $r<\min\{r_2,r_0\}$
\bee
\begin{split}
V_k(B_k(x,r))\ge& V_k(B_{g_0}(x, e^{-F(s+\tau)}r)\\
\ge& \exp(n(F(s-\tau)-F(s+\tau))(1-\delta_0)r^n\omega_n\\
\ge&(1+c_2\kappa)^{-n}(1-\delta_0)r^n\omega_n\\
\ge &(1-\delta)r^n\omega_n
\end{split}
\eee
by the choices of $\kappa, \delta_0$.
 We can conclude that there is $r_2>0$ such that for all $r<r_2$, for all $k, x\in U_k$,
\be\label{e-metric-1}
V_k(B_k(x,r))\ge (1-\delta)r^n\omega_n.
\ee
By rescaling, using Theorem \ref{t-TW} and the estimates in \cite{Chen2009,Shi1989,Simon2008} one can conclude that Theorem \ref{t-shortime} is true as in \cite{He2016}.

\end{proof}

\section{preservation of nonnegativity of  bisectional curvature}

Let $(M^n,g_0)$ be a complete noncompact \K manifold with complex dimension $n$. Let $g(t)$ smooth complete solution to  the \KR flow on $M\times[0,T], T>0$, with initial data $g_0$. That is:
\be\label{e-KR}
\left\{
  \begin{array}{ll}
  \ppt g_{\a\bb}&=-R_{\a\bb}; \\
   g(0) &= g_0.
  \end{array}
\right.
\ee
  In the proof of the uniformization result Corollary \ref{c-uniformization}, we use a result in \cite{HuangTam2015} which says that if the curvature of $ g(t)$ is bounded by $a/t$ with $a>0$ small enough depending only on $n$, then nonnegativity of bisectional curvature is preserved under the flow.  There are examples of solutions with curvature bound  as above but $a$ may not be small, see \cite{ChauLiTam1} for example. So in this section we want to remove the restriction that $a$ being small.  In fact, we will obtain the following stronger result:

\begin{thm}\label{t-nonnegative} Let $(M^n,g(t))$  be a smooth complete noncompact solution to the \KR flow on $M\times[0,T]$, $T>0$.
\begin{enumerate}
  \item [(i)] $g(0)$ has nonnegative bisectional curvature; and
  \item [(ii)] $|\Rm(g(t))|\le at^{-\theta}$   for some $a>0$, $0<\theta<2$ on $M\times(0,T]$.
\end{enumerate}
 Then $g(t)$ also has nonnegative bisectional curvature for $t\in[0,T]$. \end{thm}

 Theorem \ref{t-nonnegative} is a consequence of the following local estimate, which will also be used in the proof of Theorem \ref{t-curvature-est-1}.
 \begin{thm}\label{t-lowerbound-BK-1}
  For any $a_0>0$, $2>\theta>0$, and $1>\beta>\theta-1$, there exists $T_0(n,a_0,\theta,\beta)>0$, with the following properties: Let $(M^n,g(t))$ be a complete smooth solution to the \KR flow on $M\times[0,T], T>0$, and $x_0\in M$.  Suppose:
  \begin{enumerate}
  \item [(i)] $\mathrm{BK}(g(0))\ge 0$ on $B_0(x_0,4r)$, $r\ge1$; and
  \item [(ii)] $|\Rm(g(t))|\le at^{-\theta}$   for some $a_0\ge a>0$, $0<\theta<2$ on $B_0(x_0,4r)\times(0,T]$.
\end{enumerate}
  Then
  $$
  \mathrm{BK}(g(t))\ge -\exp\lf(-r^\frac12 t^{-\beta}\ri)
  $$
  on $B_t(x_0,r)$ for all $0<t\le T\wedge T_0$.

 \end{thm}

  Here
 $$
 a_1\wedge\cdots\wedge a_l:=\min\{a_1,\dots,a_l\}.
 $$
Also, the geodesic ball of radius $r$ with center at $x_0$ with respect to $g(t)$ will be denoted by $B_t(x_0,r)$ and its volume will be denoted by $V_t(x_0,r)$. The notation ${\rm BK}(g(t))\ge K$ means that the bisectional curvature of $g(t)$ is bounded below by $K$. That is:
$$
R(X,\bar X,Y,\bar Y)\ge K\lf(|X|^2 |Y|^2+|\la X,\bar Y\ra|^2\ri)
$$
for all $X, Y\in T^{(1,0)}(M)$.

Moreover,   a constant  $C>0$ depending only on $\a,\b,\gamma,...$
will be denoted as $ C(\a,\beta,\gamma,...)$.

Let us collect some known results on the distance function $d(t,x)$ from $x_0$ with respect to the metric $g(t)$. See \cite{Chow,Perelman-1,SimonTopping2016}.

\begin{lma}\label{l-distance} Let $(N^k,g(t)),t\in [0,T]$ be a smooth complete solution to the Ricci flow, where $N^k$ is a noncompact smooth manifold with real dimension $k$. Let $x_0\in N$ and let $d(t,x)$ be the distance function from $x_0$ with respect to $g(t)$. Suppose $|\Rm(g(t))|\le at^{-\theta}$, with $a>0, 2>\theta>0$ on $B_0(x_0,2r)$. Then
\begin{enumerate}
  \item [(i)] \bee
B_t(x_0, 2r-C(k,\theta)a^\frac12 t^{1-\frac\theta2})\subset B_0(x_0,2r).
\eee
  \item [(ii)] Let $T'>0$ be such that $2r-C(k,\theta)a^\frac12 t^{1-\frac\theta2}\ge\frac32 r$ for $t<T\wedge T'$. Then

  \bee
(\frac{\p}{\p t}-\Delta)\lf(d(x,t)+C(k,a,\theta)t^ {1-\frac12\theta}\ri)\ge 0
\eee
in the barrier sense for $x\in B_t(x_0,\frac32r)\setminus B_t(x_0, t^{\frac\theta2})$, $t< T\wedge T'$.
\end{enumerate}
\end{lma}

 Let us first recall a quadratic form for any $(0,4)$-tensor $T$ on $T_\mathbb{C}M$ with a metric $g$ by

\bee \begin{split} Q(T)(X,\bar X,Y,\bar Y):=&\sum\limits_{\mu,\nu=1}^n(|T_{X\bar\mu\nu\bar Y}|^2-|T_{X\bar\mu Y\bar\nu}|^2+T_{X\bar X\nu\bar \mu}T_{\mu\bar \nu Y\bar Y})\\
&-\sum\limits_{\mu=1}^n\textbf{Re}(T_{X\bar\mu}T_{\mu\bar XY\bar Y }+T_{Y\bar\mu}T_{X\bar X\mu\bar Y})\end{split}
\eee
for all $X,Y\in T^{(1,0)}(M)$, where $T_{\a\bb\gamma\bar\delta}=T(e_\a,\bar e_\b,e_\gamma,\bar e_\delta)$,
$T_{\a\bar \b}=g^{\gamma\bar \delta}T_{\a\bb\gamma\bar\delta}$ and $\{e_1,\dots,e_n\}$ is a unitary frame with respect to the metric $g$, $T_{X\bar \mu\nu\bar Y}=T(X,\bar e_\mu,e_\nu,\bar Y)$ etc. Here $T$ is a tensor has the following properties:
\bee
\ol{T(  X,Y,Z,W)}=T(\bar  X,\bar Y,\bar Z,\bar W);
\eee
\bee
T(X,Y,Z,W)=T(Z,W,X,Y)=T(X,W,Z,Y)=T(Y,X,W,Z).
\eee

Recall the evolution equation for holomorphic bisectional curvature: (See \cite[Corollary 2.82]{Chow})

\bee (\ppt-\triangle)R(X,\bar X,Y,\bar Y)=Q(R)(X,\bar X,Y,\bar Y) \eee for all $X,Y\in T^{1,0}_\mathbb{C}M$. Here $\triangle$  is with respect  to $g(t)$.

Next define a $(0,4)$-tensor $B$ on $T_\mathbb{C}M$ (with a metric $g$) by:
\bee
B(E,F,G,H)=g(E,F)g(G,H)+g(E,H)g(F,G) \eee
for all $E,F,G,H\in T_\mathbb{C}M$.

For positive integer $m$, let $\phi:[0,\infty)\to [0,1]$ be a smooth function such that
\be\label{e-cutoff-1}
\phi(s)=\left\{
  \begin{array}{ll}
  1&\hbox{\ for $0\le s\le \frac14$;}\\
      \exp(-\frac1{(1-s)^m})&\hbox{\ for $\frac12\le s\le 1$;} \\
    0, & \hbox{\ for $s\ge1$.}
  \end{array}
\right.
\ee
and $\phi'\le 0$.
Then
\be\label{e-cutoff-2}
\left\{
  \begin{array}{ll}
    |\phi'|\le   & C_m \phi\,|\log \phi|^{1+\frac1m};\\
    |\phi''|\le &C_m\phi\,|\log\phi|^{2(1+\frac1m)}.
  \end{array}
\right.
\ee

To prove Theorem \ref{t-lowerbound-BK-1}, we begin with the following lemma, which says that the bisectional curvature of $g(t)$ is 'nonnegative' up to high order near $t=0$.

\begin{lma}\label{l-lowerbd-1} Suppose $(M^n,g(t))$ is a complete solution to the \KR flow on $M\times[0,T], T>0$ and $x_0\in M$. Suppose $\mathrm{BK}(g(0))\ge0$ on $B_0(x_0,4r)$
  and $|\Rm(g(t))|\le E$ on $B_0(x_0,4r)\times[0,T]$. Then for any $0<\beta<1$, there is a neighborhood $U$ of $x_0$ inside $B_0(x_0,4r)$ and $T>T'>0$ such that
  $$
  \mathrm{BK}(g(t))\ge -\exp(-t^{-\beta})
  $$
  on $U\times[0,T']$.
\end{lma}
\begin{proof} We first prove that the lemma is true for any $0<\beta<\frac12$.
Since the curvature of $g(t)$ is uniformly  bounded on $B_0(x_0,4r)\times[0,T]$, by Lemma \ref{l-distance}, there is $0<T'<T$ such that $B_t(x_0,2r)\subset B_t(x_0,2r)$ for $t\le T'$ and

  \be\label{e-distance-1}
(\frac{\p}{\p t}-\Delta)\lf(d(x,t)+C_1t\ri)\ge 0
\ee
in the barrier sense for $x\in B_t(x_0,3r)$ with $d(x,t)+C_1t\ge \frac r2$ for some $C_1>0$. By choosing a smaller $T'>0$, we may assume that $2r+C_1t\le 3r$.

 Let
$$
\Psi(x,t)=\phi(r^{-1}(d(x,t)+C_1t))
$$
where $\phi$ is the function defined in \eqref{e-cutoff-1} with $m$ is chosen so that $ \beta(1+1/m)<\frac12$.

 Let $\lambda(t)=\exp(-t^{-\beta}).$ For any $\e>0$, let
\bee
\begin{split}
A_{XYZW}:=A(X,Y,Z,W) =& \Psi(x,t) R(X,Y,Z,W)+  (\lambda(t)+\e t^\frac12)  B(X,Y,Z,W)\\
\end{split}
\eee
Let $f(t)=\lambda(t)+\e t^\frac12$.

 \bee
H(x,t)=\inf\{A_{X\bar XY\bar Y}(x,t)| |X|_t=|Y|_t=1, X,Y\in T^{(1,0)}_xM\}. \eee
As in \cite{HuangTam2015}, if $H<0$ somewhere, then   one can prove that there exist     $x_1\in M$ and $T'\ge t_1>0$ such that $H(x_1,t_1)=0$ and $H(x,t)\ge0$ on $M\times[0,t_1]$. So there is $X_0, Y_0$ with norm 1 so that by extending $X_0, Y_0$ near a neighborhood of $(x_1,t_1)$ in a suitable way, which are still denoted by $X_0, Y_0$, so that if $h(x,t)=A_{X_0\bar X_0Y_0\bar Y_0}(x,t)$, then $h(x_1,t_1)=0$ and $h(x,t)\ge0$ near $(x_1,t_1)$, and $t\le t_1$. Using the evolution equation of the distance function Lemma \ref{l-distance}, we may apply the maximum principle treating $h$  as smooth. Here we use the fact that $\phi'\le0$.
Hence at  $(x_1,t_1)$,
 \be\label{e1} \begin{split}
0\ge&(\ppt-\triangle)h\\
=& \Psi \heat  R +R    \heat \Psi\\
&-2\la \nabla R,\nabla \Psi\ra+f\p_t B +f'B\\
=&{\rm I+II+III+IV+V}.
\end{split}
\ee
Here we write $B=B(  X_0,\bar{  X_0},  Y_0,\bar{  Y_0})$ etc. In the following, we simply denote $t_1$ by $t$. In the following $C_i$ will denote positive constants which are independent of $\e$. Note that at $(x_1,t_1)$, $\Psi>0$ and $R(X_0,\bar X_0,Y_0,\bar Y_0)\le 0$.

Since $Q(A)\ge0$ by the null vector condition \cite{Bando1984,Mok} and the fact that $Q(B)\le 0$ \cite{HuangTam2015}, we have
\bee
\begin{split}
0\le&\Psi^2Q(R)+C(n)\Psi |\Rm| f.
\end{split}
\eee

So
\be\label{e2}
\mathrm{I}\ge -C_2 f.
\ee
Using the \KR flow equation and the fact that the curvature is bounded, we have
\be\label{e3}
\begin{split}
\mathrm{IV}\ge& -C_2f,
\end{split}
\ee
if we choose a larger $C_2$.

\bee
\begin{split}
\mathrm{II}\ge&-C_3|R|\Psi |\log \Psi|^{2(1+\frac1m)}.
\end{split}
\eee
where we have used \eqref{e-distance-1}, the fact that $\phi'\le0$, $R\le0$ and \eqref{e-cutoff-2}.

Now $A_{X_0\bar X_0Y_0\bar Y_0}=0$ at $(x_0,t_0)$ by the upper bound of $|\Rm|$, we have
$$
E \Psi\ge c(n)f\ge c(n)\lambda(t).
$$
$$
|\log \Psi|\le C t^{-\beta}.
$$
provided that $T'$ is small but independent of $\e$.
Hence
\be\label{e4}
\begin{split}
\mathrm{II}\ge  &-C_5f t^{-2\beta(1+\frac1m)}   .\\
\end{split}
\ee
 Since $\nabla h=0$ at $(x_1,t_1)$,  and $\nabla B=0$, we have

 \be\label{e5}
\begin{split}
\mathrm{III}\ge &-2|R|\frac{|\nabla\Psi|^2}{\Psi} \\
&\ge-C_5f t^{-2\beta(1+\frac1m)}  .\\
\end{split}
\ee
with a possible larger $C_5$.

\be\label{e6}
\mathrm{V}\ge C_6(\e t^{-\frac12}+\beta t^{-\beta-1}\lambda)
\ee

By \eqref{e1}--\eqref{e6}, we have

\bee
\begin{split}
0\ge& \e t^\frac12\lf(C_6t^{-1}-2C_5t^{-2\beta(1+\frac1m)}-2C_2\ri)+\lambda\lf(C_6\beta t^{-\beta-1}-2C_5t^{-2\beta(1+\frac1m)}-2C_2\ri)
\end{split}
\eee
Hence if $T'>0$ is small enough but independent of $\e$, this is a contradiction. Let $\e\to0$, we conclude that
$$
 \mathrm{BK}(g(t))\ge  -\Psi^{-1}\exp(-t^{-\beta})
$$
in $U\times(0,T']$ for some neighborhood of $x_0$ at the point where $\Psi\neq0$. Since this is true for all $\beta<\frac12$, by choosing a smaller $T'$, we conclude that the lemma is true for any $\beta<\frac12$.

The proof also implies that for any compact set $K$ in $B_0(x_0,4r)$ and for any $0<\beta<\frac12$, there is $T'>0$ so that $\mathrm{BK}(g(t))>-\exp(-t^{-\beta})$ for $T'>t>0$ in $K$.

For any $\beta<\frac12(1+\frac12)$ choose $\sigma<\frac12$  so that $1+\sigma> 2\beta$ and choose $m$ large, so that $2\beta(1+\frac1m)<1+\sigma$ and $2\beta(1+\frac1m)<\beta+1$. In the above computations, if we replace $\e t^\frac12$ with $\e \exp(-t^{-\sigma})$, then one can repeat the argument so that if $H<0$ somewhere, then there is a $(x_1,t_1)$ with $t_1>0$ so that
\bee
\begin{split}
0\ge&\e \exp(-t^{-\sigma})\lf(C_7t^{-1-\sigma}-C_8t^{-2\beta(1+\frac1m)}-C_8\ri)\\
&+
\exp(-t^{-\beta})\lf(C_7t^{-1-\beta}-C_8t^{-2\beta(1+\frac1m})-C_8\ri)
\end{split}
\eee
for some positive constants $C_7, C_8$ independent of $\e$. So we have a contradiction if $T'>0$ is small enough by the choice of $m, \sigma$.  Hence the lemma is true for $\beta<\frac12(1+\frac12).$

Continue in this way, we conclude that the lemma is true for $\beta<\sum_{i=1}^l 2^{-i} $ for any $l\ge1$. This implies the lemma is true.

\end{proof}

 \begin{proof}[Proof of Theorem \ref{t-lowerbound-BK-1}]
  In the following:

\begin{itemize}
  \item $c_i$ denote positive constants depending only on $n$.
  \item $C_i$ denote positive constants depending only on $n, a_0, \theta, \beta$.
\end{itemize}

  Let $1>\beta>\theta-1$. Choose $m$ large in \eqref{e-cutoff-1} so that
  \be
  \beta+1>2\beta(1+\frac1m).
  \ee
  This can be done because $\beta<1$. By Lemma \ref{l-lowerbd-1}, we conclude that for any compact set $K$ inside $B_0(x_0,4r)$, there is $\hat T>0$ such that on $K\times[0,\hat T]$,
  $$
   \mathrm{BK}(g(t))\ge -\exp(-r^\frac12 t^{-\beta})
   $$
   because $r\ge1$. First choose $T'(n,a_0,\theta)>0$ so that for $0<t\le T'\wedge T$, we have
   $B_t(x_0,3r)\subset B_0(x_0,4r)$ and such that
   \bee
   (\frac{\p}{\p t}-\Delta)(d(x,t)+C_1t^{1-\frac\theta2})\ge0
   \eee
   in the barrier sense for $x\in B_t(x_0,3r)$ with $d(x,t)+C_1t^{1-\frac\theta2}\ge r$.

   As in the proof of Lemma \ref{l-lowerbd-1}, \bee
\begin{split}
A_{XYZW}:=A(X,Y,Z,W) =& \Psi(x,t) R(X,Y,Z,W)+   \lambda(t)   B(X,Y,Z,W)\\
\end{split}
\eee
where $\lambda(t)=\exp(-r^{\frac12}t^{-\beta})$ and
$$
\Psi(x,t)=\phi\lf((2r)^{-1}(d(x,t)+C_1t^{1-\frac\theta2})\ri).
$$
Since $\phi'(s)=0$ for $0\le s\le\frac12$ and $\phi'\le 0$, we have
\be\label{e-distance-2}
 (\frac{\p}{\p t}-\Delta)\Psi\le C_2r^{-2}\Psi |\log \Psi|^{2(1+\frac1m)}
\ee
by \eqref{e-cutoff-2}.

Let
 \bee
H(x,t)=\inf\{A_{X\bar XY\bar Y}(x,t)| |X|_t=|Y|_t=1, X,Y\in T^{(1,0)}_xM\}. \eee

Then $H(x,t)>0$ for on $B_0(x_0,3r)$ if $t$ is small enough.

Argue as in the proof of Lemma \ref{l-lowerbd-1}, if $H(x,t)<0$ for some $x\in M$ and $t>0$, then there is a point $(x_1,t_1)$ with $t_1>0$ and $\Psi(x_1,t_1)>0$ so that

\be\label{ee1} \begin{split}
0\ge & \Psi \heat  R +R    \heat \Psi\\
&-2\la \nabla R,\nabla \Psi\ra+\lambda\p_t B +\lambda'B\\
=&{\rm I+II+III+IV+V}.
\end{split}
\ee
As before, denote $t_1$ simply by $t$ we have

\be\label{ee2}
\mathrm{I}\ge -C_2t^{-\theta}\lambda.
\ee
Using the \KR flow equation and the fact that the curvature is bounded, we have
\be\label{e3}
\begin{split}
\mathrm{IV}\ge& -C_2t^{-\theta}\lambda,
\end{split}
\ee
if we choose a larger $C_2$.

Using \eqref{e-distance-2} and the fact that $R\le 0$ at $(x_1,t_1)$, we have
\bee
\begin{split}
\mathrm{II}\ge&-C_3r^{-2}|R|\Psi |\log \Psi|^{2(1+\frac1m)}.
\end{split}
\eee

Now $A_{X_0\bar X_0Y_0\bar Y_0}=0$ at $(x_1,t_1)$ by the upper bound of $|\Rm|$, we have
$$
t^{-\theta} \Psi\ge C_4\lambda.
$$
$$
|\log \Psi|\le C_4r^\frac12 t^{-\beta}
$$
with a larger $C_4$, provided that $T'=T'(n,a_0,\theta)>0$ is small.
Hence
\be\label{ee4}
\begin{split}
\mathrm{II}\ge  &-C_5r^{-2+(1+\frac1m)} t^{-2\beta(1+\frac1m)} \lambda    \\
\ge& -C_5t^{-2\beta(1+\frac1m)})\lambda
\end{split}
\ee
because $r\ge1$.

 One also have

 \be\label{ee5}
\begin{split}
\mathrm{III}\ge &-C_5  t^{-2\beta(1+\frac1m)} \lambda .\\
\end{split}
\ee
with a possibly larger $C_5$.

\be\label{ee6}
\mathrm{V}\ge C_6r^\frac12 \beta t^{-1-\beta}\lambda.
\ee
 Combining \eqref{ee1}--\eqref{ee6}, we have at $(x_1,t_1)$,
 \be\label{ee7}
 \begin{split}
 0\ge&\lambda\lf(C_6r^\frac12\beta t^{-1-\beta}-2C_2t^{-\theta}-2C_5t^{-2\beta(1+\frac1m)} \ri)
 \end{split}
 \ee
 Since $\beta+1>\theta$, and $\beta+1>2\beta(1+\frac1m)$, we conclude that there is a $T_0(n,a_0,\theta,\beta)$ (because $m$ depends on $\beta$) so that if $t\le T_0\wedge T$, then \eqref{ee7} is impossible.  Since $\Psi>0$ if $d(x,t)+C_1 t^{1-\frac\theta2}<2r$, we conclude that the theorem is true provided that $T_0(n,a_0,\theta,\beta)$ is small, because $r\ge 1$.

 \end{proof}

\begin{proof}[Proof of Theorem \ref{t-nonnegative}]  Fix $T_0\wedge T\ge t>0$ and let $r\to \infty$ in Theorem \ref{t-lowerbound-BK-1}, where $T_0$ is as in the theorem, we have
$$
\mathrm{BK}(g(t))\ge0
$$
for $t\le T_0\wedge T.$ For $T\ge t>T_0\wedge T$ one can use the standard result in \cite{Shi1997} or just iterate Theorem \ref{t-lowerbound-BK-1} to conclude that $\mathrm{BK}(g(t))$ is also nonnegative.
\end{proof}

\section{curvature estimates}

In this section, we want to  discuss     curvature bound in Theorem \ref{t-nonnegative}.

Let $(M^n,g_0)$ be a complete noncompact \K manifold with complex dimension $n$. Let $g(t)$ be a smooth complete solution to  the \KR flow on $M\times[0,T], T>0$, with initial data $g(0)=g_0$. Let $V_t(x,r)$ be the volume of geodesic ball of radius $r$ with center at $x$ with respect to the metric $g(t)$. In this section we  will prove the following:

\begin{thm}\label{t-curvature-est-1} Let $(M^n,g(t))$ be as above.
  Suppose

  \begin{enumerate}
    \item [(i)] $\mathrm{BK}(g(0))\ge0$;
    \item [(ii)] $V_0(x,1)\ge v_0>0$ for all $x\in M$ for some $v_0>0$; and
    \item [(iii)] the curvature of $g(t)$ is uniformly bounded on $M\times(\tau, T]$ for all $\tau>0$.
  \end{enumerate}
     Then $|\Rm(g(t))|\le \displaystyle{\frac at}$  on $M$ for some $a>0$ for all $t\in (0,T]$. In particular, $\mathrm{BK}(g(t))\ge0$ for all $t\in (0,T]$.

\end{thm}

Theorem \ref{t-curvature-est-1} is a direct consequence of the following pseudolocal result and Theorem \ref{t-nonnegative}.
\begin{thm}\label{t-curvature-est-2}
 Let $(M^n,g(t))$ be a smooth  complete solution to the \K Ricci flow on $M\times [0,T]$ with $T>0$. Let $x_0\in M$. Suppose
\begin{enumerate}
  \item [(i)] $V_0(x,r)\ge v_0r^{2n}>0$ for all $x\in B_0(x_0,3)$ and $r\le 2$;
  \item [(ii)] $\mathrm{BK}(g(0))\ge 0$ on $B_0(x_0,8)$;
  \item [(iii)] for any $T>\tau>0$, the curvature of $g(t)$ is uniformly bounded on $M\times[\tau,T]$.
\end{enumerate}
Then there exist $a_0(n,v_0)>0$, $T^{**}(n,v_0)>0$ such that
$$
|\Rm(g(t))|\le \displaystyle{\frac {a_0}t}
$$
on  $B_t(x_0,1)$ for all $0<t\le T\wedge T^{**}$.
\end{thm}
\begin{rem} Condition (i) can be replaced by $V_0(x_0,1)\ge v_0$ by volume comparison. In this case,  (i) will be true with $v_0$   replaced by $c(n)v_0$ for some $c(n)>0$.
\end{rem}

We need the following lemma, which was basically proved by Simon-Topping \cite{SimonTopping2016}. Their result is on Ricci flow on a Riemannian three manifold using the result of Perelman \cite{Perelman-1} on nonflat ancient solution. In our case, we just appeal to a result by Ni \cite{Ni2005} which states that  every complete noncompact nonflat ancient solutions to the \KR flow with nonnegative and bounded bisectional curvature cannot have maximum volume growth.
\begin{lma}\label{l-curvest-1}
Suppose that $(M^{n},g(t))$ is a smooth complete solution to the \K Ricci flow for $t\in [0,T]$, $T>0$.   Suppose further there exist  $K\geq 0,\, v_0>0$ such that
\begin{enumerate}
  \item [(i)]  $V_0(x_0,r)\geq v_0r^{2n}$ for $r\le 2$; and
  \item [(ii)] $\mathrm{BK}(g(t))\geq -K$ on $B_t(x_0,2)$ for all $t\in [0,T]$.
\end{enumerate}
Then   there exist $ T_1(n,K,v_0)>0$ and $\Theta(n,K,v_0)>0$ such that for all $0<t\le T\wedge T_1$,
$$ |Rm(g(t))|\le \frac{\Theta}{t}
$$
on $B_t(x_0,1).$
\end{lma}
\begin{proof} The proof is exactly the same as the proof of Lemma 2.1 in \cite{SimonTopping2016}, except in our case, we will use the result in \cite{Ni2005} as described above.
\end{proof}
\begin{cor}
Suppose that $(M^{n},g(t))$ is a smooth complete solution to the \K Ricci flow for $t\in [0,T]$, $T>0$.   Suppose further there exist  $K\geq 0,\, v_0>0$ such that
\begin{enumerate}
  \item [(i)]  $V_0(x,1)\geq v_0$ for all $x\in M$; and
  \item [(ii)] $\mathrm{BK}(g(t))\geq -K$ on $M$ for all $t\in [0,T]$.
\end{enumerate}
Then   there exist $ T_1(n,K,v_0)>0$ and $\Theta(n,K,v_0)>0$ such that for all $0<t\le T\wedge T_1$,
$$ |Rm(g(t))|\le \frac{\Theta}{t}.
$$
\end{cor}

Given Theorem \ref{t-lowerbound-BK-1}, Lemma \ref{l-curvest-1} and Lemma 2.3 in \cite{SimonTopping2016} (see also \cite{Simon2012}).  One can proceed as in \cite{SimonTopping2016}. Let us first prove the following:

\begin{lma}\label{l-curvest-2} For any $v_0>0$, $\Theta>0$, there is $\e(n,v_0)>0$ and $T_2(n, v_0,\Theta)>0$ with the following property: Suppose  $(M^n,g(t))$ is a smooth complete noncompact solution to the \KR flow on $M\times[0,T]$ with $0<T\le T_2$ and $x_0\in M$ such that
\begin{enumerate}
  \item [(i)] $|\Rm(g(t))|\le \Theta/t$ on $B_t(x_0,1)$, $t\in (0,\a^2T]$;
  \item [(ii)] $V_0(x_0,r)\ge  v_0r^{2n}$  for all $r\le \frac12$; and
  \item [(iii)] $|Rm(g(t))|$ is uniformly bounded on $M\times(\tau, T]$ for all $\tau>0$.
  \item [(iv)] $\mathrm{BK}(g(0))\ge 0$ on $B_0(x_0,8)$.
\end{enumerate}
Then $|\Rm(x_0,T)|\le (\e\a)^{-2}\Theta/T$, where $0<\a<1$ is defined by the relation $$\e^2\a^2=(1-\a^2)\Theta.$$
\end{lma}

We need the following pseudolocality theorem of Perelman \cite{Perelman-1} as in the form of \cite[Theorem 6.2]{SimonTopping2016}:
\begin{thm}\label{t-pseudolocal}
For any $\delta>0$, and positive integer $m >0$, there exists $\e > 0$
such that if $r_0 > 0$ and $(M^m, g(t)))$ is a complete solution to the Ricci flow on $M\times[0,\e^2r_0^2]$,   with the properties that $|\Rm (g(0))|\le r_0^{-2}$ on $B_0(x_0,r_0)$  and
$V_0 (x_0, r_0)\ge \delta r_0^{m}$,
then $|\Rm|(g(t))\le (\e r_0)^{-2}$ for any  $0\le t\le \e^2r_0^2$ and
throughout $B_t(x_0,\e r_0)$.
\end{thm}
\begin{proof}[Proof of lemma \ref{t-curvature-est-2}] By \cite[Lemma 2.3]{SimonTopping2016}, there is a constant $\delta(n,v_0)>0$ such that if $g(t)$ is a complete solution to the \KR flow satisfying curvature bound in (i) and is such that $\Ric(g(t))\ge -n$   on $B_t(x_0,1)$, then for $t$ small enough then for $0<r_0\le \frac12$
\be\label{e-volume}
V_t(x_0,r_0)\ge \delta r_0^{2n}.
\ee
We should emphasis that $\delta$ does not depend on $\Theta$. Let  $\e=\e(n,\delta_0)=\e(n,v_0) $  be the positive constant in Theorem \ref{t-pseudolocal} so that the lower bound of the volume satisfies \eqref{e-volume} and $m=2n$.

 Apply Theorem \ref{t-lowerbound-BK-1} for    $\theta=1$, and $\beta=\frac12$, we can find   $1> T_0(n,\Theta,k)>0$ such that if $T\le T_0$ and if $g(t)$ satisfies (i) and (iv), then
$$
\mathrm{BK}(g(t))\ge -1
$$
on $B_t(x_0,\frac12)$ for all $t\in (0,\a^2T]$.   By \cite[Lemma 2.3]{SimonTopping2016}, there exists $T_2(n,\Theta,v_0)>0$ such that if $T<T_2(n,\Theta,v_0)$ then
$$
V_t(x_0,r)\ge \delta r^{2n}
$$
for all $t\in (0,\a^2T]$ and $0<r\le \frac12$. We may assume that $T_2\le \frac14\e^2\wedge T_0$. Then we have the following situation: $V_{\a^2 T}(x_0,r)\ge \delta r^{2n}$ for $r\le \frac12$, $|\Rm (g(\a^2T))|\le \Theta/(\a^2T) $ on $B_{g(\a^2T)}(x_0,1)$. Let $r_0>0$ be such that $$ \e^2 r_0^2= \frac{\e^2\a^2T}{\Theta}=(1-\a^2)T<T\le \frac14\e^2. $$
Then $r_0\le \frac12$ and $r_0^{-2}=\Theta/(\a^2 T)$. Apply   Theorem \ref{t-pseudolocal} to the \KR flow $g(t)$ on $[\a^2T, T]=[\a^2T, \a^2T+\e^2r_0^2]$, we conclude that
$$
|\Rm(g(x_0,T)|\le (\e r_0)^{-2}=\frac{1}{(1-\a^2)T}=\frac1{\e^2\a^2}\cdot\frac{\Theta}{T}
$$
because $\e^2\a^2=(1-\a^2)\Theta$.
\end{proof}

\begin{proof}[Proof of Theorem \ref{t-curvature-est-2}]\noindent

\begin{itemize}

  \item  Let   $ T_1(n,v_0), \Theta(n,v_0)>0$ be the constants in Lemma \ref{l-curvest-1} with   $K=1$.
      \item Let $\e(n,v_0)$, $\a(n,v_0)$  and $T_2(n,v_0,\Theta)$ be the constants in Lemma \ref{l-curvest-2}.

          \item Let $T_0(n,v_0)>0$   be the constants in Theorem \ref{t-lowerbound-BK-1}  with $a_0=((\e \a)^{-2}+1)\Theta$, $r=2$ and $\theta=1$, $\beta=\frac12$.

\end{itemize}
Let $T^*=T_0 \wedge T_1\wedge T_2$. Then $T^*=T^*(n,v_0)$.

 By \cite[Lemma 4.2]{SimonTopping2016}, there is $\beta(n)>0$ so that for all $L>0$, one of the following is true:

\begin{enumerate}
  \item [ \underline{\bf Case 1}:] For all $t\in(0,T)$ with $t\le \displaystyle{\frac{9}{\beta^2 a_0(L+1)^2}}$, we have
      $$B_t(x_0,3-(L+1)\beta\sqrt {a_0t})\subset  B_0(x_0,3)$$
      and
      $$
      |\Rm(g(t))|\le \frac{a_0}t
      $$
      on $B_t(x_0,3-(L+1)\beta\sqrt {a_0t})\supset B_t(x_0,2)$.

  \item [\underline{\bf Case 2}:]  There exist  $0<t_0<T$ with $t_0\le \displaystyle{\frac{9}{\beta^2 a_0(L+1)^2}}$ and $z_0$   in the closure of $B_{t_0}(x_0,3-(L+1)\beta\sqrt {a_0t_0})$ such that
      \begin{enumerate}
        \item [(a)] $|\Rm(z_0,t_0)|=\displaystyle{\frac{a_0}{t_0}}$;
        \item [(b)] $|\Rm(g(t))|\le \displaystyle{\frac{a_0}{t_0}}$ on $B_t(z_0,L(1-\a)\beta\sqrt {a_0t_0})$ for all $t\in (0,\a^2 t_0]$; and
        \item [(c)] for all $t\in (0,\a^2 t_0]$
        $
        B_t(z_0,L(1-\a)\beta\sqrt {a_0t_0})\subset B_0(x_0,3).
        $
      \end{enumerate}

\end{enumerate}

 Fix $L(n,v_0)>0$ such that
      \be\label{e-L-1}
      \displaystyle{\frac{9}{\beta^2 a_0(L+1)^2}}\le T^*;\  \ L(1-\a)\beta\sqrt{a_0T^*}\ge 8.
      \ee
      Let $$T^{**}=\min\{T^*, \displaystyle{\frac{9}{\beta^2 a_0(L+1)^2}}\}= \displaystyle{\frac{9}{\beta^2 a_0(L+1)^2}}.$$

       To prove the theorem, we may assume that $T\le T^{**}$.

We want to prove {\bf Case 2} in the above cannot happen.  Suppose $t_0, z_0$ is as in {\bf Case 2}.
Let
$$
h(y,s)=\frac{T^*}{t_0}g(y, \frac{t_0}{T^*}s).
$$
$0\le s\le \frac{TT^*}{t_0}$.
Let $\wt \Rm$ be the curvature tensor of $h$ etc. Then
\begin{enumerate}
  \item [(i)] $|\wt\Rm(z_0,T^*)|=\displaystyle{\frac{a_0}{T^*}=\frac{\lf((\e\a)^{-2}+1\ri)\Theta}{T^*}}$;
  \item [(ii)] $|\wt\Rm(z_0,s)|\le  \displaystyle{\frac{a_0}{s}}$ in $\wt B_s(z_0, 8)$ for all $0<s\le \a^2T^*$ because of \eqref{e-L-1} .
  \item [(iii)] $\mathrm{BK}(h(0))\ge 0$ on $\wt B_0(z_0,8)$.
  \item[(iv)] $\wt V_0(z_0,r)\ge v_0r^{2n}$ for all $r\le 2$.
\end{enumerate}
By (ii), (iii) and Theorem \ref{t-lowerbound-BK-1}, we conclude that $\mathrm{BK}(h(s))\ge-1$ on $\wt B_s(z_0,2)$ for $0<s\le \a^2T^*$.
By (iv) and Lemma \ref{l-curvest-1}, we conclude that $|\wt\Rm(h(s))|\le \Theta/s$ on $\wt B_s(z_0,1)$ for all $s\in (0,\a^2 T^*]$. By Lemma \ref{l-curvest-2}, this implies
  $$
 | \wt \Rm(z_0,T^*)|\le \displaystyle{\frac{ (\e\a)^{-2} \Theta}{T^*}}
 $$
 which contradicts   (i) in the above.

 Hence {\bf Case 1} must be true. By the choice of $L$ in \eqref{e-L-1}, we have
 $$
 3-(L+1)\beta\sqrt{a_0t}\ge2
 $$
 for $t\le T^{**}$.
 This completes the proof of the theorem.
\end{proof}

\end{document}